\documentclass{article}
\usepackage{amsfonts}
\usepackage{amsmath}
\usepackage{amsthm}
\usepackage{hyperref}
\newcommand{\E}{\mathbb{E}}
\newcommand{\e}{\epsilon}
\newcommand{\df}{\mathrm{d}}
\newtheorem{theorem}{Theorem}
\newtheorem{claim}{Claim}
\newtheorem{lemma}{Lemma}
\newtheorem{conjecture}{Conjecture}

\begin{document}
\title{A Variant of The Corners Theorem}
\author{Matei Mandache\footnote{Mathematical Institute, University of Oxford. E-mail: \href{mailto:matei.mandache@maths.ox.ac.uk}{matei.mandache@maths.ox.ac.uk}
}}
\maketitle
\begin{abstract}
The Corners Theorem states that for any $\alpha > 0$ there exists an $N_0$ such that for any abelian group $G$ with $|G| = N \geq N_0$ and any subset $A \subset G \times G$ with $|A| \ge \alpha N^2$ we can find a corner in $A$ , i.e. there exist $x, y, d \in G$ with $d \neq 0$ such that $(x, y), (x+d, y), (x, y+d) \in A$.

Here, we consider a stronger version: given such a group $G$ and subset $A$, for each $d \in G$ we define $S_d = \{(x, y) \in G \times G : (x, y), (x+d, y), (x, y+d) \in A \}$ . So $|S_d|$ is the number of corners of size $d$. Is it true that, provided $N$ is sufficiently large, there must exist some $d \in G \setminus \{0\}$ such that $|S_d|> (\alpha^3 - \epsilon ) N^2$ ?

We answer this question in the negative. We do this by relating the problem to a much simpler-looking problem about random variables. Then, using this link, we show that there are sets $A$ with $|S_d| < C\alpha^{3.13} N^2$ for all $d \neq 0$, where $C$ is an absolute constant. We also show that in the special case where $G = \mathbb{F}_2^n$, one can always find a $d$ with $|S_d|> (\alpha^4 - \epsilon ) N^2$.
\end{abstract}
\section{Introduction}

In 1975, Szemer\'{e}di proved the following famous result about arithmetic progressions (see \cite{szemeredi75})

\begin{theorem}
Let $\e > 0$ and let $k$ be a positive integer. Then there exists a $N_0 = N_0(k, \e)$ such that whenever $N \ge N_0$ and $A \subset [N]$ has $|A| \ge \e N$, then $A$ contains a non-trivial arithmetic progression of length $k$
\end{theorem}

Here we are using $[N]$ to denote the set $\{1, 2, \dots, N\}$. There are several ways one could try and strengthen this result. In \cite{bergelsonhostkra05}, Bergelson, Host and Kra conjecture that, instead of finding just one arithmetic progression, one can find a large number of arithmetic progressions all with the same common difference. More specifically, they make the following conjecture.

\begin{conjecture}\label{BHK}
Let $0 < \alpha < 1$, let $\e > 0$ and let $k = 3$ or $4$. Then there exists an $N_0$ such that for all $N \ge N_0$ and any $A \subset [N]$ with $|A| \ge \alpha N$ there must exist a $d \neq 0$ such that the set $A$ contains at least $(\alpha^k - \e)N$ distinct arithmetic progressions of length $k$ and with common difference $d$.
\end{conjecture}

Note that if we choose $A$ at random by putting each element of $[N]$ in $A$ independently with probability $\alpha$, then each arithmetic progression is in $A$ with probability $\alpha^k$. It follows that the bound $(\alpha^k - \e)N$ is best possible. In the same paper, Bergelson, Host and Kra also provide a counter-example showing that the conjecture is false when $k=5$. Conjecture \ref{BHK} is now a theorem. Indeed, in 2005 Green \cite{ARLoriginal} proved the $k=3$ case of the above conjecture, by first  proving an arithmetic regularity lemma. This arithmetic regularity lemma is a kind of analogue of Szemer\'{e}di's graph regularity lemma, in the setting of abelian groups instead of graphs. In 2010 Green and Tao \cite{ARLadvanced} proved a more powerful arithmetic regularity lemma, and used it to prove Conjecture~\ref{BHK} for $k=4$ as well as $k=3$.

A different way to generalise Szemer\'{e}di's theorem on arithmetic progressions is to look for higher-dimensional patterns in subsets of $[N]^d$. This is known as the Multidimensional Szemer\'{e}di Theorem. The Corners Theorem is a special case of this.

\begin{theorem}\label{cornerstheorem} [Corners Theorem]
Let $\e > 0$. There exists an $N_0=N_0(\e)$ such that for any $N \ge N_0$ and any $A \subset [N]^2$ with $|A| \ge \e N^2$, there exist $x, y \in [N]$ and an integer $d \neq 0$ such that $(x, y), (x+d, y)$ and $(x, y+d)$ are all in $A$. The triple of points $((x, y), (x+d, y), (x, y+d))$ is called a corner.
\end{theorem}

The Corners Theorem was first proved by Ajtai and Szemer\'{e}di in 1974 \cite{ajtaiszemeredi74}. Now, several proofs exist, see for example \cite{shkredov06} and \cite{furstenberg81}. In light of the Bergelson-Host-Kra conjecture, it is natural to make the following conjecture.

\begin{conjecture}\label{benshope}
Let $0<\alpha<1$ and $\e>0$. There exists an $N_0$ such that for any $N \ge N_0$ and any $A \subset [N]^2$ with $|A| \ge \alpha N^2$, there is some integer $d \neq 0$ such that there are at least $(\alpha^3-\e)N^2$ corners with that given $d$ all of whose vertices lie in $A$.
\end{conjecture}

We will show that this conjecture is in fact false. Indeed, there is a construction in which there are at most $C\alpha^{3.13}N^2$ corners for any given $d \neq 0$, where $C$ is an absolute constant. Chu \cite{chu11} considered an ergodic theory analogue of the above problem, and proved that the exponent 3 is not right in that setting. Our construction is broadly the same as Chu's, however her construction does not traslate directly to our setting and we need to use concentration of measure arguments. Our construction is also an improvement on Chu's in that we get a better bound on the exponent of $\alpha$.

We will be working with abelian groups, rather than the set $[N]$, but one can easily deduce the result for $[N]$ by embedding it into $\mathbb{Z}/N\mathbb{Z}$ in the usual way.

For a family of abelian groups $(G_n)_{n=1}^\infty$ with $|G_n|$ tending to infinity, we define the threshold density for corners $M(\alpha)$ as follows.

$M(\alpha)$ is the largest $M \in [0,1]$ such that for any $\e > 0$ there is an $N_0$ such that for all groups $G$ in our family with $|G| = N \ge N_0$, and for all subsets $A \subset G \times G$ of density at least $\alpha$, there is a $d \in G \setminus \{ 0 \}$ with $|S_d|>(M - \epsilon) N^2$.

Over the course of this paper, we relate the problem of finding the threshold density for corners to the following problem about random variables.

Let $f$ be a piecewise constant function on $[0,1]^3$, taking values in $[0,1]$. By piecewise constant, we mean that there is a positive integer $m$ and real numbers $0 = x_0 \le x_1 \le \dots \le x_m = 1$, $0 = y_0 \le y_1 \le \dots \le y_m = 1$ and $0 = z_0 \le z_1 \le \dots \le z_m = 1$ such that $f$ is constant on $[x_i, x_{i+1}) \times [y_j, y_{j+1}) \times [z_k,z_{k+1})$ for each $0 \le i,j,k \le m-1$. We define
\[ T(f) := \int_0^1 \int_0^1 \int_0^1 \left( \int_0^1 f(x, y, z') \df z' \right) \left( \int_0^1 f(x, y', z) \df y' \right) \left( \int_0^1 f(x', y, z) \df x' \right) \df x \df y \df z\]
and
\[ m(\alpha) := \inf_f T(f) \]
where the infimum is taken over all piecewise constant $f: [0,1]^3 \to [0,1]$ with
\[ \int_0^1 \int_0^1 \int_0^1 f(x, y, z) \df x \df y \df z = \alpha. \]
Note that if we let $X, Y$ and $Z$ be independent $U[0,1]$ random variables, and setting $f = f(X, Y, Z)$, we have
\[ T(f) = \E(\E(f|X, Y)\E(f|X, Z)\E(f|Y, Z)) \]
and
\[ m(\alpha) = \inf_f T(f) \]
where the infimum is taken over all piecewise constant $f$ with $\E(f) = \alpha$. 

Our main results are the following:

\begin{theorem}\label{pt1}
For the family of all abelian groups, $M(\alpha) \le m(\alpha)$.
\end{theorem}
\begin{theorem}\label{pt2}
Let $m':[0,1] \to [0,1]$ be any convex function of $\alpha$ satisfying $m'(\alpha) \le m(\alpha)$ for all $\alpha \in [0,1]$. For the family $G_n = \mathbb{F}_2^n$, $M(\alpha) \ge m'(\alpha)$.
\end{theorem}

Note that if $m(\alpha)$ is itself convex, then we may take $m'=m$ and combine the two theorems to get $M(\alpha) = m(\alpha)$ for $G_n = \mathbb{F}_2^n$.
In Section~\ref{construction} we prove Theorem~\ref{pt1}, using a random construction.
In section~\ref{regularity} we use regularity techniques to prove Theorem~\ref{pt2}.  The method generalises to the case $G_n = \mathbb{F}_p^n$, with $p$ any fixed prime.
Finally in section~\ref{bounds} we prove the bounds $\alpha^4 \le m(\alpha) \le C\alpha^{3.13}$ for an absolute constant $C$, and so deduce that $M(\alpha) \le C\alpha^{3.13}$ in the general case, and $M(\alpha) \ge \alpha^4$ in the finite field case.

We will work with an alternative formulation of the corners problem, which is more natural in some respects. Instead of working in $G \times G$, we work in the subset $P \subset G^3$ given by $P =\{(x, y, z) \in G^3 | x+y+z = 0\}$, and a corner is a triple of the form $(x+d, y, z), (x, y+d, z), (x, y, z+d)$. This is clearly seen to be equivalent to the original formulation, by forgetting one of the coordinates. We use this formulation because it treats the three vertices of the corner equally, whereas in the original formulation, the vertex $(x, y)$ was special.

We will also need the following Lemma, which is essentially giving us an alternative version of the random variable problem.
\begin{lemma}\label{discrete}
Given a piecewise constant $f : [0, 1]^3 \to [0, 1]$, we can construct independent random variables $X', Y'$ and $Z'$ which each take only finitely many values, and a function $f'=f'(X', Y', Z')$ taking values in $[0,1]$ with $\E(f') = \E(f)$ and $T(f') = T(f)$.

Conversely, given a trio of random variables $X', Y', Z'$ each taking only finitely many values, and a function $f' = f'(X', Y', Z')$ taking values in $[0,1]$, we may construct a piecewise constant $f : [0,1]^3 \to [0,1]$ with $\E(f) = \E(f')$ and $T(f) = T(f')$.
\end{lemma}
\begin{proof}
\underline{First part:} Suppose we are given such a function $f$. Let $m$, a positive integer, and $0 = x_0 \le x_1 \le \dots \le x_m = 1$, $0 = y_0 \le y_1 \le \dots \le y_m = 1$ and $0 = z_0 \le z_1 \le \dots \le z_m = 1$ be such that $f$ is constant on $[x_i, x_{i+1}) \times [y_j, y_{j+1}) \times [z_k,z_{k+1})$ for each $0 \le i,j,k \le m-1$. Given $X, Y, Z \sim U[0,1]$ independent we define new random variables $X', Y'$ and $Z'$ taking values in $[m]$ by setting $X'$ to be the $i$ such that $X \in [x_{i-1}, x_i)$, $Y'$ to be the $j$ such that $Y \in [y_{j-1}, y_j)$ and $Z'$ to be the $k$ such that $Z \in [z_{k-1}, z_k)$. The independence of $X', Y'$ and $Z'$ follows from the indepence of $X, Y$ and $Z$. By the piecewise constant assumption, the value of the function $f$ is determined by the values of $X', Y'$ and $Z'$, so there exists a function $f' : [m]^3 \to [0,1]$ such that $f(X, Y, Z) = f'(X', Y', Z')$. The first part follows easily.

\underline{Second part:} Now suppose we are given indepent random variables $X', Y'$ and $Z'$, each taking only finitely many values. Without loss of generality, we may assume that $X', Y'$ and $Z'$ take values in $[m]$ for some positive integer $m$. This is because the actual values taken by $X', Y'$ and $Z'$ are unimportant, so we can simply relable them. We then define real numbers $0 = x_0 \le x_1 \le \dots \le x_m = 1$, $0 = y_0 \le y_1 \le \dots \le y_m = 1$ and $0 = z_0 \le z_1 \le \dots \le z_m = 1$ by setting $x_i = \mathbb{P}(X' \le i), y_i = \mathbb{P}(Y' \le i)$ and $z_i = \mathbb{P}(Z' \le i)$. Then $\mathbb{P}(X' = i) = x_i - x_{i-1} = \mathbb{P}(X \in [x_{i-1}, x_i))$ so we may couple $X$ with $X'$ in the same way as we did in the first part. The same applies, of course, for $Y$ with $Y'$ and $Z$ with $Z'$. We can then construct a piecewise constant $f : [0,1]^3 \to [0,1]$ with $f(X, Y, Z) = f'(X', Y', Z')$, and so the second part follows.
\end{proof}
\section{Construction}\label{construction}
In what follows, we will need some properties of the functions $m$ and $M$, which we will establish now.

\begin{lemma}\label{continuity}
Let $m' : [0,1] \to [0,1]$ be convex with $m'(\alpha) \le m(\alpha)$ for all $\alpha \in [0,1]$. $M$, $m$ and $m'$ are all increasing functions of $\alpha$, and $m$ is uniformly continuous.
\end{lemma}
\begin{proof} $M$ and $m$ are trivially increasing. $m(0) = 0$ so $m'(0) = 0$, which implies $m'$ is increasing by convexity. So we focus on showing that $m$ is uniformly continuous.
Given any piecewise constant function $f: [0,1]^3 \to [0,1]$, we can form a new function $f_\e$ by setting $f_\e = \e + (1-\e)f$. This function will have $\E(f_\e) = \e + (1-\e)\E(f) = \E(f) + \e(1-\E(f))$, and
\begin{align*}
T(f_\e) & = \E(\E(f_\e|X, Y)\E(f_\e|X, Z)\E(f_\e|Y, Z)) \\
& \le \E((\E(f|X, Y)+\e)(\E(f|X, Z)+\e)(\E(f|Y, Z)+\e)) \\
& \le T(f) + 3\e + 3 \e^2 + \e^3.
\end{align*}
The last inequality above is obtained by expanding and using the trivial bound $f \le 1$. By taking $f$ to have $\E(f) = \alpha$ and $T(f)$ approaching $m(\alpha)$, we deduce $m(\alpha) \le m(\alpha + \e(1-\alpha)) \le m(\alpha) + 3\e + 3 \e^2 + \e^3$, which implies continuity of $m$ at all points except from 1. To show continuity at 1, we use the following inequality: 
\[T(f) \ge \E(\E(f|X, Y)+\E(f|X, Z)+\E(f|Y, Z)-2) = 3\E(f)-2\]
so $m(\alpha) \ge 3\alpha-2$, and $m$ is therefore continuous at 1. $m$ is compactly supported and continuous, therefore it is uniformly continuous.
\end{proof}

We will use a probabilistic construction to prove Theorem~\ref{pt1}. We let $\e_1, \e_2, \e_3 > 0$ and assume $N$ is sufficiently large in terms of $\e_1, \e_2$ and $\e_3$. We take a piecewise constant function $f: [0,1]^3 \to [0,1]$ with $E(f)= \alpha$ and  $T(f)= m_1$, where $m_1 < m(\alpha) + \e_1$ (this is possible by the definition of $m$). We then construct our subset as follows: for each $g \in G$, we take independent $U[0,1]$ random variables $X_g, Y_g, Z_g$. We then include each point $(x, y, z) \in P$ in our set $A$ independently with probability $f(X_x, Y_y, Z_z)$. Because $\E(f) = \alpha$, this means our subset $A$ will have expected size $\alpha N^2$.

Now, let $(x+d, y, z), (x, y+d, z), (x, y, z+d)$ be an arbitrary corner (with $d \neq 0$). The probability that this corner is contained in $A$ is 
\[ \E(f(X_{x+d},Y_y, Z_z)f(X_x, Y_{y+d}, Z_z)f(X_x, Y_y, Z_{z+d})) .\]
The random variables $X_{x+d}, Y_{y+d}, Z_{z+d}$ are each independent of all the other random variables appearing here, so the above probability is in fact
\[ \E(\E(f|X_x, Y_y)\E(f|X_x, Z_z)\E(f|Y_y, Z_z))=T(f)=m_1 .\]
We would like to argue that the size of $A$ must, with high probability, be close to its expected value, and similarly for the size of the set of corners of given $d$.

\subsection{Size of $A$}
We say a subset $Q \subset P$ is independent if for each pair of distinct points $ p_1 =(x_1, y_1, z_1), p_2 =(x_2, y_2, z_2)$ in $Q$, $x_1 \neq x_2, y_1 \neq y_2$ and $z_1 \neq  z_2$. This condition is equivalent to the set of events $\{ p \in A | p \in Q\}$ being mutually independent. Now, let $Q$ be an independent set of size at least $N/3$. It is easy to see such a set must exist. Indeed, any maximal independent set has size at least $N/3$. Also note that, for any $p \in P$, $Q+p$ will also be an independent set.

Now, for any $p \in P$ the size of $(Q+p) \cap A$ is the sum of $|Q|\ge N/3$ independent Bernoulli random variables, so it is distributed as the binomial distribution $B(|Q|, \alpha)$. By concentration of measure arguments (see, for example, Hoeffding \cite{hoeffding63}) we know that for binomial distributions $X \sim B(n, p)$ the tail probabilities $\mathbb{P}(|X-np|/n > c)$ decrease exponentially with $n$. Therefore $\mathbb{P}((Q+p) \cap A < (\alpha - \e_2)|Q|) < 1/3N^2$ for $N$ sufficiently large. So, with probability at least $2/3$, $|(Q+p) \cap A|>(\alpha - \e_2)|Q|$ for all $p \in P$. By averaging over $p \in P$, we have $|A|>(\alpha - \e_2)N^2$.
\subsection{Size of $S_d$}
We follow a similar method here, with a couple of modifications. For a given $d \neq 0$, we can consider $S_d$, the set of corners of size $d$, to be a subset of $P_d = \{(x, y, z) \in G^3 | x+y+z+d = 0\}$, by setting $(x, y, z) \in S_d$ if the corner $(x+d, y, z), (x, y+d, z), (x, y, z+d)$ is in $A$. We now define a subset $Q \subset P_d$ to be independent if for each pair of distinct points $ p_1 =(x_1, y_1, z_1), p_2 =(x_2, y_2, z_2)$ in $Q$, $x_1 \notin \{x_2, x_2+d, x_2-d \} $ and similarly for the other coordinates. This condition ensures that the set of events $\{ p \in S_d | p \in Q\}$ are mutually independent. For each $d$, we can take an independent $Q_d \subset P_d$ with $|Q_d| \ge N/9$, and for all $p \in P$, $Q_d+p$ will also be independent.

Since for all $p \in P$ the size of $(Q_d+p) \cap S_d$ is the sum of $|Q_d|\ge N/9$ independent Bernoulli random variables, it is binomially distributed and therefore $\mathbb{P}((Q_d+p) \cap S_d > (m_1 + \e_3)|Q_d|) < 1/3N^3$ for $N$ sufficiently large. So, for each $d$, with probability at least $1-1/3N$, we have $|(Q_d+p) \cap S_d|<(m_1 + \e_3)|Q_d|$ for all $p \in P$. By averaging over $p \in P$, we have $|S_d|<(m_1 + \e_3)N^2$. This happens simultaneously for every $d \in G\setminus \{0\}$ with probability at least $2/3$.

Putting the two estimates together, we see that with probability at least $1/3$, we have that both $|A|>(\alpha - \e_2)N^2$ and $|S_d|<(m_1 + \e_3)N^2<(m(\alpha) + \e_1 + \e_3))N^2$. So, if we set $\alpha' = \alpha - \e_2$, we have $M(\alpha') \le m(\alpha' + \e_2) + \e_1 + \e_3$. By taking $\e_1, \e_2$ and $\e_3$ to be sufficiently small, and using the fact that $m$ is continuous, we get that $M(\alpha') \le m(\alpha')$, as required.
\section{Regularity}\label{regularity}
Throughout this section, we set $G$ to be the group $\mathbb{F}_2^n$.
We define a boxing of $G$ to be the following data:
A subgroup $W \subset \mathbb{F}_2^n$, an integer $m$, and for each coset $V = W^3 + (x, y, z)$ with $(x, y, z) \in P$, partitions of each of $W+x, W+y, W+z$ into $m$ sets: $B_1, B_2, \dots, B_m$ $C_1, C_2, \dots, C_m$ and $D_1, D_2, \dots, D_m$. These cosets $V$ will be refered to as ``outer boxes''. There are $4^{\text{codim}(W)}$ of them, and they form a partition of $P + W^3$.

We will write $(B_i \times C_j)_P$ for the set $\{ (x, y, z) \in P | x \in B_i, y \in C_j \}$. Note that this set can be identified with $B_i \times C_j$ by forgetting the third coordinate. Similarly we write $(B_i \times D_j)_P$ for the set $\{ (x, y, z) \in P | x \in B_i, z \in D_j \}$ and $(C_i \times D_j)_P$ for the set $\{ (x, y, z) \in P | y \in C_i, z \in D_j \}$. We will refer to sets of this type as ``inner boxes''. For each outer box $V$, these inner boxes give us three partitions of the set $P \cap V$.

We will need some Fourier analysis. Given an abelian group $W$, let $W^* = \hom (W, S^1)$ be the dual group and $(g)^\wedge$ and $\hat{g}$ both denote the Fourier transform of a function $g$, defined as $\hat{g}(r) = \E_{x \in W} g(x)r(x) $ for $r \in W^*$. When $S$ is a subset of the group $W$, we will also use $S$ to denote the indicator function of the set $S$. If $B_i \subset W+x$, we define Fourier coefficients for the set $B_i$ to be the values of $\widehat{B_i - x}$ (and the same for $C_i$ and $D_i$). The size of these Fourier coefficients is independent of the choice of $x$. 

We will also need the notion of quasirandomness from graph theory. This is the notion used in Szemer\' edi's graph regularity lemma \cite{szemeredi76}. By identifying $(B_i \times C_j)_P$ with $B_i \times C_j$, $A$ restricted to the inner box gives us a bipartite graph between $B_i$ and $C_j$. We say that $A$ is $\e$-quasirandom if for any $B' \subset B_i$ and $C' \subset C_i$ with $|B'| \ge \e|B_i|$ and $|C'| \ge \e|C_i|$, the density of $A$ on $B' \times C'$ differs from the density of $A$ on $B_i \times C_j$ by at most $\e$.

A boxing will be called regular if it satisfies the following two regularity conditions:
\begin{itemize}
\item [(1)] For all but a proportion of at most $\e$ of the outer boxes $V$, all of the associated sets $B_i, C_j, D_k$ are $\e/m^3$ uniform, in the sense that all of their non-trivial Fourier coefficients are bounded in size by $\e/m^3$.
\item [(2)] For all but a proportion of at most $\e$ of the outer boxes $V$, if we choose $(x, y, z) \in P \cap V$ uniformly at random, then for each of the three inner boxes containing $(x, y, z)$ (i.e. $(B_i \times C_j)_P$, $(B_i \times D_k)_P$, and $(C_j \times D_k)_P$ for some $i, j, k$), the probability that the set $A$ is $\e$-quasirandom in the inner box is at least $1-\e$.
\end{itemize}
Given a regular boxing, we will call an outer box $V$ regular if it is not one of the exceptional boxes for (1) or (2).

The main result of this section is:

\begin{theorem}\label{regularitylemma} [Regularity Lemma]
Given $\e>0$, there exists an $n_1$ and $m_1$ such that if $n \ge n_1$, and $A \subset G = \mathbb{F}_2^n$, then there exists a regular boxing where $W$ has codimension at most $n_1$ and the parameter $m$ is at most $m_1$.
\end{theorem}

In the proof that follows, we will not obtain any effective bounds on $n_1$ or $m_1$. The method of proof goes via an energy increment argument, similar to that used in the proof of Green's regularity lemma \cite{ARLoriginal}. However, since we have two different regularity conditions to satisfy, we will be using two energies, one associated with each type of regularity. We define these energies now.

Let $\delta(B_i) := |B_i|/|W|$ be the density of $B_i$ inside $W+x$ and define $\delta(C_i), \delta(D_i)$ similarly. Let $\delta_{B_i, C_j} := |A \cap (B_i \times C_j)_P|/|B_i||C_j|$ be the density of $A$ in the inner box $(B_i \times C_j)_P$, and define $\delta_{B_i, D_j}, \delta_{C_i, D_j}$ similarly.
\[ E_1 = \frac{1}{4^{\text{codim}(W)}}\sum_{\text{cosets} V}\sum_{i=1}^{m}\frac{(\delta(B_i)^2+\delta(C_i)^2+\delta(D_i)^2)}{3} \]

\[ E_2 = \frac{1}{4^{\text{codim}(W)}}\sum_{\text{cosets} V}\sum_{i,j=1}^{m}\frac{1}{3}\big(\delta(B_i)\delta(C_j)\delta^2_{B_i, C_j}+ \]
\[\delta(B_i)\delta(D_j)\delta^2_{B_i, D_j}+\delta(C_i)\delta(D_j)\delta^2_{C_i, D_j}\big) \]

Note that $E_1$ and $E_2$ take values in $[0,1]$, so they cannot be incremented indefinitely.

We will show that:

\begin{itemize}
\item [(A)] If regularity condition (1) fails to hold, we may refine the boxing so that $E_1$ is incremented by some amount depending on $\e$ and $m$ only, without increasing $m$ or decreasing $E_2$.

\item [(B)] If regularity condition (2) fails to hold, we may refine the boxing so that $E_2$ is incremented by an amount depending on $\e$ only.
\end{itemize}

These two statements together imply that, starting with any boxing, we may refine the boxing until we get a regular boxing. Indeed, we have the following algorithm:

\begin{itemize}
\item[i.] If (2) doesn't hold, then apply (B) until it does hold.
\item[ii.] If (2) holds but (1) doesn't, apply (A) until (2) does hold. Then, if (1) doesn't hold, go back to i. Otherwise, both conditions hold and so we are done.
\end{itemize}

This must terminate because the number of times $E_2$ can be incremented in total is bounded, and the number of times $E_1$ can be incremented in a row is also bounded (with the bound depending on the parameter $m$).

\subsection{Refinement (A)}

We assume regularity condition (1) fails to hold. Suppose that $V$ is an outer box, and that some $B_i, C_i$ or $D_i$ associated with $V$ is not $\e/m^3$-uniform. This will happen to a proportion of at least $\e$ of the outer boxes. Without loss of generality, we may assume $B_1$ is non-uniform.

$B_1 \subset W + x$, so $B_1 - x \subset W$. The non-uniformity means that we can find some $\xi \in W^* \setminus \{0\}$ such that $|\widehat{(B_1 - x)}(\xi)| \ge \e/m^3$. Set $W'$ to be the kernel of $\xi$. This will be a codimension-1 subspace of $W$, and so we can write $W = W' \cup (W' + v)$ for some $v$.

Each of the $B_i$ can be decomposed as $B_i = B'_i \cup B''_i$ where $B'_i = B_i \cap (W'+x)$ and $B''_i = B_i \cap (W' + v +x)$ (and similarly for $C_i$ and $D_i$). This gives us a way of refining our boxing by replacing the original subspace $W$ by a codimension 1 subspace $W'$. For each of the four cosets of $W'^3$ in the original coset $V$, we get the three new partitions by restricting the old partitions to the relevant coset of $W'$, so the new $B_1, \dots, B_m$ associated to a coset of $W'^3$ will either be $ B'_1, \dots, B'_m$ or $B''_1, \dots, B''_m$. Now we inspect what happens to $E_1$. If we let $\delta(B'_i) = |B'_i|/|W'|$ and $\delta(B''_i) = |B''_i|/|W'|$, then each term $\delta(B_i)^2$ will be replaced by

\begin{align*} \frac{1}{4}(2\delta(B'_i)^2+2\delta(B''_i)^2) & = \left( \frac{\delta(B'_i) + \delta(B''_i)}{2}\right)^2 + \left( \frac{\delta(B'_i) - \delta(B''_i)}{2}\right)^2 \\
& = 
 \delta(B_i)^2 + \left( \frac{\delta(B'_i) - \delta(B''_i)}{2}\right)^2 ,
\end{align*}
and similarly for $\delta(C_i)^2$ and $\delta(D_i)^2$. 

Since we haven't yet used any special properties of $W'$, this shows that $E_1$ does not decrease when we refine the boxing by replacing $W$ with any codimension~1 subspace. By iterating, this means that $E_1$ will not decrease when we replace $W$ by any subspace and restrict the sets $B_i, C_j$ and $D_k$ in the natural way to get a boxing.
Now, $|\widehat{(B_1 - x)}(\xi)| \ge \e/m^3$ means that $|\delta(B'_1) - \delta(B''_1)| \ge \e/m^3$, so replacing $W$ with $W'$ will increment $E_1$ by at least

\[ \frac{1}{4^{\text{codim}(W)}} \frac{\e^2}{12m^6} .\]

Note that this increment can also be achieved by replacing $W$ with any subspace of $W'$. This increment still depends on $\text{codim}(W)$, so it is not enough on its own. However, we can apply the same method simultaneously for each outer box $V$ which does not satisfy the first regularity condition. There are at least $\e 4^{\text{codim}(W)}$ such outer boxes. For each outer box $V$ which fails (1), we get a codimension~1 subspace $U_V$ of $W$. We set $T = \cap_V U_V$. The refinement we will take is the refinement obtained by replacing $W$ with $T$ and restricting the sets $B_i, C_j$ and $D_k$ in the natural way to get a boxing. Since $T$ is a subspace of $U_V$ for all outer boxes $V$ which fail (1), $E_1$ will be incremented by at least
\[\frac{\e^3}{12m^6} ,\]
which depends only on $m$ and $\e$, as required. Also note that the parameter $m$ does not change when we apply this refinement. However, we still need to check that $E_2$ has not been decreased.

To this end, the following lemma comes in useful.
\begin{lemma}\label{incrementformula}
Suppose we are given a set $S$, a subset $M \subset S$ with $|M| = \mu |S|$, and a partition $S_1, S_2, \dots, S_t$ of $S$. Then
\[ \sum_{i=1}^t \frac{|S_i|}{|S|} \left( \frac{|M \cap S_i|}{|S_i|} \right)^2 \ge \mu^2 .\]
If moreover there exists an $i$ with $|S_i| \ge c_1 |S|$, $\left| |M \cap S_i|/|S_i| - \mu \right| \ge c_2$, then
\[ \sum_{i=1}^t \frac{|S_i|}{|S|} \left( \frac{|M \cap S_i|}{|S_i|} \right)^2 \ge \mu^2 + c_1c_2^2 \]
\end{lemma}
\begin{proof}
Let $X$ be a random variable defined as follows: take $p \in S$ uniformly at random, and set $X = |M \cap S_i|/|S_i|$ where $i$ is such that $S_i \ni p$. Then we have
\[ \E(X) = \sum_{i=1}^t \frac{|S_i|}{|S|} \frac{|M \cap S_i|}{|S_i|} = \sum_{i=1}^t \frac{|M \cap S_i|}{|S|} = \mu \]
and
\[ \E(X^2) = \sum_{i=1}^t \frac{|S_i|}{|S|} \left( \frac{|M \cap S_i|}{|S_i|} \right)^2 .\]
So
\[ \sum_{i=1}^t \frac{|S_i|}{|S|} \left( \frac{|M \cap S_i|}{|S_i|} \right)^2 = \mu^2 + \operatorname{Var}(X)  \]
The first part follows from non-negativity of variance. In the second part, we may assume $|S_1| \ge c_1 |S|$, $||A \cap S_1|/|S_1| - \mu| \ge c_2$. We have 
\[ \operatorname{Var}(X) \ge c_2^2 \mathbb{P}( |X-\E(X)|\ge c_2) \ge c_2^2 \mathbb{P}(p \in S_1) \ge c_1c_2^2, \]
which implies the result.
\end{proof}

Now let us look at $E_2$. $E_2$ is simply the weighted sum of the square of the density of $A$ in each inner box. One can check that in the refinement described above, each old inner box is replaced by a family of new inner boxes which partition it. So, by setting $S$ to be the old inner box, $M$ to be $A \cap S$ and $S_1, \dots S_t$ to be the relevant new inner boxes in Lemma~\ref{incrementformula}, we can see that $E_2$ is not decreased in the process.

\subsection{Refinement (B)}

If regularity condition (2) fails to hold, we will use a different kind of refinement. The subspace $W$ and the outer boxes will stay the same, but we will refine the three partitions $B_1, \dots, B_m$, $C_1, \dots C_m$ and $D_1, \dots D_m$. The first part of Lemma~\ref{incrementformula} tells us that $E_2$ will not decrease when we do this. Let $V$ be an outer box which does not satisfy (2), and let $(B_i \times C_j)_P$ be an inner box which is not $\e$-quasirandom. Then we can find $B' \subset B_i$, $C' \subset C_j$ with $|B'| \ge \e |B_i|$, $|C'| \ge \e|C_j|$ with $|\delta_{B' \times C'}-\delta_{B_i \times C_j}| \ge \e$. So if we refine our partitions by replacing $B_i$ with $B'$ and $B_i \setminus B'$ and replacing $C_j$ with $C'$ and $C_j \setminus C'$, we can apply the second part of Lemma~\ref{incrementformula} using

\[ S = (B_i \times C_j)_P, M = A \cap S,\]\[S_1 = (B' \times C')_P, S_2 = (B' \times C_j \setminus C')_P, S_3 = (B_i \setminus B' \times C')_P, S_4 = (B_i\setminus B' \times C_j\setminus C' )_P ,\]

Since $|S_1 \ge \e^2 |S|$ and $\left||M \cap S_1|/|S_1|  - \mu \right| \ge \e $, $E_2$ will be incremented by at least
\[ \frac{1}{4^{\text{codim}(W)}} \frac{\delta(B_i) \delta(C_j) \e^4}{3} .\]

We can apply the same process simultaneously to all of the inner boxes inside $V$ which are not $\e$-quasirandom (note that such boxes make up a proportion of at least $\e$ of $V$'s inner boxes), and also do the same for each outer box $V$ which fails (2) (of which there are at least $\e 4^{\text{codim}(W)}$). This gives us a total increment in $E_2$ of at least $\e^6/3$. Note that, if we wish to ensure that the parameter $m$ is the same throughout the boxing, we may simply add empty sets to some of the partitions until they all have equal size.

We can now deduce that a regular boxing exists. To get the quantitative result of Theorem~\ref{regularitylemma}, we note that in each of the refinements, the parameters $\text{codim}(W)$ and $m$ increase in a controlled way.

\subsection{Counting}

Now that we have proved Theorem~\ref{regularitylemma}, we will use it to count corners. Let $m'(\alpha)$ be a convex function with $m'(\alpha) \le m(\alpha)$ for all $\alpha \in [0,1]$. Let $G = \mathbb{F}_2^n$ be sufficiently large so that we can apply the regularity lemma (Theorem~\ref{regularitylemma}) with $\e$. Suppose $A \subset P$ has size $|A| \ge \alpha N^2$. We take an $\e$-regular boxing for $A$ (this gives us the data $W \le G$, together with three partitions $B_1, \dots, B_m$, $C_1, \dots, C_m$ and $D_1, \dots, D_m$ for every outer box $V = W^3 + p$ with $p \in P$). Throughout what follows, we will use $o(1)$ to denote some quantities that tend to 0  as $\e \to 0$. The main result of this section is the following:

\begin{claim}\label{countclaim}
The number of corners lying in $A$ with $d \in W$ is at least $(m'(\alpha) + o(1))N^2|W|$. 
\end{claim}

Note that for a corner, the condition $d\in W$ is equivalent to requiring the three vertices of the corner to be in the same outer box. By the pigeonhole principle, it follows from  the above claim that there is some $d \in W \setminus \{0\}$ with $|S_d| \ge (m'(\alpha) + o(1) + O(|W|^{-1}))N^2$. Note that, for a fixed $\e$, $|W|$ tends to infinity with $N$, so the $O(|W|^{-1})$ term can be made as small as we like. Therefore $M(\alpha) \ge m'(\alpha)$, which is Theorem~\ref{pt2}.

Now we set about proving Claim~\ref{countclaim}. For each outer box $V$, we define $\alpha(V)$ to be the density of $A$ inside $V$, $\alpha(V) = |V \cap A|/|V \cap P|$. The average value of $\alpha(V)$ over all $V$ is at least $\alpha$. We will show that if $V$ is a regular outer box, then the number of corners in $V$ is at least $(m(\alpha(V) + O(\e)) + o(1))|W|^3$ (this is the content of Claim~\ref{1box} below). Since $m(\alpha)$ is uniformly continuous, $(m(\alpha(V) + O(\e)) + o(1))|W|^3 = (m(\alpha(V))+ o(1))|W|^3 \ge (m'(\alpha(V))+ o(1))|W|^3$. Then, we can use the convexity of $m'$ to take the sum over all $V$, and so the number of corners with $d \in W$ is at least $ (m'(\alpha) + o(1))N^2|W|$ . Note that here we use the fact that the proportion of irregular boxes is $o(1)$, so the error that comes from those boxes is $o(1)$.

So, it is sufficient to prove the following claim:

\begin{claim}\label{1box}
If $V$ is a regular outer box, then there are at least $(T(f) + o(1))|W|^3$ corners in $V$, where $f$ is some function of independent random variables $X, Y$ and $Z$, where $X, Y$ and $Z$ each take finitely many values, $f$ takes values in $[0,1]$ and $\E(f) = \alpha(V) + O(\e)$.
\end{claim}

By Lemma~\ref{discrete}, in the above claim, we may equivalently have assumed $f$ was piecewise constant $:[0,1]^3 \to [0,1]$ and $X, Y$ and $Z$ where $U[0,1]$.

\begin{proof}[Proof of Claim \ref{1box}]
Let $V$ be a regular outer box. Without loss of generality, we may assume $V = W^3$.
The corners whose vertices lie in $V \cap P$ are in bijection with the elements of $V$ as follows: $(x, y, z)$ correponds to the corner $K(x, y, z) := \{(x+d, y, z), (x, y+d, z), (x, y, z+d) \}$ where $d = -(x+y+z)$. For each $1 \le i, j, k \le m$, there are $|B_i||C_j||D_k|$ corners $K(x, y, z)$ with $(x, y, z) \in B_i \times C_j \times D_k$. The three vertices of these corners will lie in the inner boxes $(C_j \times D_k)_P, (B_i \times D_k)_P$ and $(B_i \times C_j)_P$ respectively. We may form a tripartite graph with vertex sets $B_i, C_j$ and $D_k$, by adding an edge between $x \in B_i$ and $y \in C_j$ if $(x, y, -x-y) \in A$, and similarly for the other two pairs. Triangles in this graph correspond to corners lying in $A$. Now, if the three inner boxes $(B_i \times C_j)_P, (B_i \times D_k)_P$ and $(C_j \times D_k)_P$ are all $\e$-quasirandom, then we may apply the counting lemma from the proof of the triangle removal lemma to conclude that the number of such triangles is $|B_i||C_j||D_k|(\delta_{B_i, C_j}\delta_{B_i, D_k}\delta_{C_j, D_k} + o(1))$ (The statement of this counting lemma is Proposition~1.2 in \cite{ARLoriginal}. The triangle removal lemma, or rather a weaker form of it, was first proved by Ruzsa and Szemer\'{e}di in 1976 \cite{ruzsaszemeredi78}). Without quasirandomness, we still know the number of corners is at most $|B_i||C_j||D_k|$. Because $V$ is a regular box, the total size of all the non-quasirandom inner boxes is at most $3\e|W|^2$, so the total of $|B_i||C_j||D_k|$ over all $i, j, k$ for which some inner box is not $\e$-quasirandom is at most $3 \e |W|^3$

Therefore, the number of corners of $V$ with vertices lying in $A$ is

\begin{equation} \label{eq:noofcorners}
\sum_{1 \le i, j, k \le m} |B_i||C_j||D_k|\delta_{B_i, C_j}\delta_{B_i, D_k}\delta_{C_j, D_k} + o(1)|W|^3 .
\end{equation}

Now we define the random variables $X, Y, Z$ and the function $f$. For any $(x, y, z) \in W^3$, define $B((x, y, z))$ to be the $B_i$ that contains $x$, $C((x, y, z))$ to be the $C_j \ni y$ and $D((x, y, z))$ the $D_k \ni z$. Let $p \in W^3$ be chosen uniformly at random, and let $X=B(p), Y=C(p), Z=D(p)$. $X, Y$ and $Z$ are independent random variables, each taking only finitely many values. Define
\[ f' = \frac{|A \cap X \times Y \times Z||W|}{|X||Y||Z|} \]
and
\[ f = \min(f', 1) .\]

One can easily check $\E(f') = \alpha(V)$. $\E(f' |X, Y) = |A \cap (X \times Y)_P|/|X||Y| = \delta_{X, Y}$ Similarly $\E(f' | X, Z) = \delta_{X, Z}$ and $\E(f' | Y, Z) = \delta_{Y, Z}$. Putting this together, we get

\[ T(f') = \sum_{1 \le i, j, k \le m} \frac{|B_i||C_j||D_k|\delta_{B_i, C_j}\delta_{B_i, D_k}\delta_{C_j, D_k}}{|W|^3} .\]

Comparing this with equation (\ref{eq:noofcorners}), we see that the number of corners is $(T(f')+o(1))|W|^3$. The problem with $f'$, though, is that it may take values greater than 1. We need to bound $f'$. Below, we will be using convolutions of functions on $W$, which are defined as follows: $g * h (y) := \E_{x \in W} g(x)h(y-x)$. 

\begin{align*}
|A \cap B_i \times C_j \times D_k| &\le |P \cap B_i \times C_j \times D_k| \\
&= |\{(x, y, z) \in B_i \times C_j \times D_k | x+y+z = 0\}| \\ &=  |W|^2 B_i * C_j * D_k (0) \\ &= |W|^2 \sum_{r \in W^*} (B_i * C_j * D_k)^\wedge(r) \\ &= |W|^2 \sum_{r \in W^*} \hat{B}_i(r) \hat{C}_j(r) \hat{D}_k(r)
\end{align*}

We split the above sum into a ``main'' term, and error terms, which can be bounded using the uniformity assumption, plus applications of the Cauchy-Schwarz inequality and Parseval's identity.

\[\sum_{r \in W^*} \hat{B}_i(r) \hat{C}_j(r) \hat{D}_k(r) = \delta(B_i)\delta(C_j)\delta(D_k) + \sum_{r \in W^* \setminus \{0\}} \hat{B}_i(r) \hat{C}_j(r) \hat{D}_k(r) \]
\[ \le \delta(B_i)\delta(C_j)\delta(D_k) + \max_{r \in W^* \setminus \{0\}} |\hat{B}_i(r)| \left|\sum_{r \in W^* \setminus \{0\}} \hat{C}_j(r) \hat{D}_k(r) \right|\]
\[ \le \delta(B_i)\delta(C_j)\delta(D_k) + \frac{\e}{m^3} \left(\sum_{r \in W^*} |\hat{C}_j(r)|^2 \right)^{1/2} \left(\sum_{r \in W^*} |\hat{D}_k(r)|^2 \right)^{1/2}\]
\[ \le \delta(B_i)\delta(C_j)\delta(D_k) + \frac{\e}{m^3} \delta(C_j)^{1/2} \delta(D_k)^{1/2} \le  \delta(B_i)\delta(C_j)\delta(D_k) + \frac{\e}{m^3} \]

We deduce that

\[ f'(B_i, C_j, D_k) \le \frac{|W|^3(\delta(B_i)\delta(C_j)\delta(D_k) + \e/m^3)}{|B_i||C_j||D_k|} = 1 + \frac{\e}{m^3 \delta(B_i) \delta(C_j) \delta(D_k)}\]
and hence
\[ \E(f' - f) \le \sum_{1 \le i, j, k, \le m} \frac{\e}{m^3} \le \e .\]

Therefore $\E(f) = \E(f') + O(\e) = \alpha(V) + O(\e)$. Using the fact that $\E(f' | X, Y) = \delta_{X, Y} \le 1, \E(f' | X, Z) = \delta_{X, Z} \le 1$ and $\E(f' | Y, Z) = \delta_{Y, Z} \le 1$, we also have
\begin{align*}
T(f') - T(f) & = \E(\E(f'-f|X, Y)\E(f'|X, Z)\E(f'|Y, Z)) \\
& + \E(\E(f|X, Y)\E(f'-f|X, Z)\E(f'|Y, Z)) \\
& + \E(\E(f|X, Y)\E(f|X, Z)\E(f'-f|Y, Z)) \\
& \le \E(\E(f'-f|X, Y)) + \E(\E(f'-f|X, Z)) + \E(\E(f'-f|Y, Z)) \\
& = 3 \E(f'-f) \le 3\e
\end{align*}
So $T(f) = T(f') + o(1)$. Combining this with previous results, we get that the number of corners is $(T(f)+o(1))|W|^3$, where $\E(f)=\alpha(V) + O(\e)$, as required.
\end{proof}

This completes the proof of Claim~\ref{1box}, and therefore also of Claim~\ref{countclaim} and Theorem~\ref{pt2}.

\section{Bounds}\label{bounds}

In this section, we prove some bounds on $m(\alpha)$, and hence get some bounds on $M(\alpha)$, the number of corners we can expect to find with the same $d$. 

Recall that $m(\alpha)$ was defined to be the infimum of
\[ T(f) := \E(\E(f|X, Y)\E(f|X, Z)\E(f|Y, Z)) \]
over all possible piecewise constant functions $f:[0,1]^3 \to [0,1]$ with $\E(f) = \alpha$.

\subsection{Upper bound}
In section 5 of \cite{chu11}, Chu proves that for all $c > 0 $ there exists an $\alpha$ such that $m(\alpha) < c\alpha^3$. Her construction in fact shows that $m(\alpha) < C\alpha^{3.086}$ for some constant $C>1$. Her method involves taking a simple example of a function $f$ with $T(f) < \E(f)^3$, and then taking tensor products. We use the same method, except that our starting function is a bit simpler and we get a better bound.

Let $X, Y, Z \sim U[0,1]$ be independent. We define the piecewise constant function $g$ as follows: 
\[ g(X, Y, Z) = \begin{cases}
    0       & \quad \text{if } (X,Y,Z) \in [0, 1/2)^3 \\
            & \quad \text{or}  (X,Y,Z) \in [1/2, 1)^3 \\
    1       & \quad \text{otherwise.}
  \end{cases}
  \]
A simple calculation shows us that $\E(g) = 3/4$ and 
\[ T(g) = \frac{13}{32} = \left(\frac{3}{4}\right)^3\left(\frac{26}{27}\right) = \left(\frac{3}{4}\right)^{3+c} ,\]
where \[ c:= \frac{\log(26/27)}{\log(3/4)} = 0.131 \dots\]
So this shows that $m(\alpha) \le \alpha^{3+c}$ for $\alpha = 3/4$. We extend this using the construction detailed below.

We construct a new piecewise constant function $f^{\otimes n}$ from an old piecewise constant function $f$ essentially by taking the $n$-fold tensor power. However, in order to define the tensor power, we need to go to the alternative formulation given by Lemma~\ref{discrete}.
Given random variables $X, Y, Z \sim U[0,1]$, and $f : [0,1]^3 \to [0,1]$, we apply the first part of Lemma~\ref{discrete} to obtain independent random variables $X', Y', Z'$, each taking only finietly many values, and a function $f'=f'(X', Y', Z')$ such that $\E(f')=\E(f)$ and $T(f') = T(f)$. We then define $X'^{\otimes n} := (X'_1, X'_2, \dots , X'_n)$ where $X'_1, X'_2, \dots X'_n$ are independent $\sim X'$, and define $Y'^{\otimes n}, Z'^{\otimes n}$ similarly. Then define 
\[ f'^{\otimes n} := \prod_{i=1}^n f'(X'_i, Y'_i, Z'_i) .\]
$f'^{\otimes n}$ is a function of $X'^{\otimes n}, Y'^{\otimes n}$ and $Z'^{\otimes n}$, and $X'^{\otimes n}, Y'^{\otimes n}$ and $Z'^{\otimes n}$ each take finitely many values.
Therefore, we can apply the second part of Lemma~\ref{discrete} to get a function $f^{\otimes n} : [0,1]^3 \to [0,1]$ with $\E(f^{\otimes n}) = \E(f'^{\otimes n}) = \E(f')^n = \E(f)^n$ and $T(f^{\otimes n}) = T(f'^{\otimes n}) = T(f')^n = T(f)^n$.

\begin{theorem} For all $\alpha \in [0,1]$, we have
\[ m(\alpha) \le \frac{27}{26}\alpha^{3+c} \]
\end{theorem}
\begin{proof} If $\alpha = 0$ this theorem is trivial, so assume $\alpha > 0$. We may write $\alpha = \beta (3/4)^k$ for some integer $k$ and $\beta \in (3/4, 1]$. Then we take $f$ to be the function $\beta g^{\otimes k}$. $\E(\beta g^{\otimes k}) = \beta \E(g)^k = \beta (3/4)^k = \alpha$, and $T(\beta g^{\otimes k}) = \beta^3 T(g^{\otimes k}) = \beta^3 (3/4)^{k(3+c)} \le (27/26) \beta^{3+c} (3/4)^{k(3+c)} = (27/26) \alpha^{3+c}$ as required.
\end{proof}
\subsection{Lower bound}

In this subsection, we will prove the lower bound $m(\alpha) \ge \alpha^4$. For convenience, we will write $f_3 = \E(f|X, Y), f_2= \E(f|X, Z) $ and $ f_1=\E(f|Y, Z)$.

Firstly, note that $\E(f \mathbf{1}(f_i < \alpha/4)) \le \alpha/4$ for $i = 1, 2, 3$ (where $\mathbf{1}$ denotes the indicator function of an event), because by the law of total expectation
\begin{align*}
\E(f \mathbf{1}(f_1 < \alpha/4)) & = \E(\E(f \mathbf{1}(f_1 < \alpha/4)|Y, Z)) \\ 
& = \E(\E(f| Y, Z) \mathbf{1}(f_1 < \alpha/4))) \\
&= \E(f_1 \mathbf{1}(f_1 < \alpha/4)) \\
& \le \alpha/4
\end{align*}
And similarly for the others. Therefore 
\[ \E(f \mathbf{1}(f_1 \ge \alpha/4)\mathbf{1}(f_2 \ge \alpha/4)\mathbf{1}(f_3 \ge \alpha/4)) \ge \E(f) - \sum_{i=1}^3  \E(f \mathbf{1}(f_i < \alpha/4)) \ge \alpha/4 \]
Therefore the event $A = (f_1 \ge \alpha/4)\wedge(f_2 \ge \alpha/4)\wedge(f_3 \ge \alpha/4)$ has probability at least $\alpha/4$.

So we have, for any $f$

\[ T(f) = \E(f_1f_2f_3) \ge \E(f_1f_2f_3|A)\mathbb{P}(A) \ge (\alpha/4)^3(\alpha/4) = \frac{\alpha^4}{256} .\]

Now suppose for contradiction some function $h$ with $\E(h)=\alpha$ has $T(h)<\alpha^4$. So $T(h)=\alpha^4/C$, some $C>1$. Pick an integer $k$ such that $C^k > 256$. Then $\E(h^{\otimes k}) = \alpha^k$ but $T(h^{\otimes k}) = \alpha^{4k}/C^k < \alpha^{4k}/256$, contradicting the above inequality for $f = h^{\otimes k}$. Therefore $m(\alpha) \ge \alpha^4$.

\subsection{Conclusion}
Now that we have bounds on $m(\alpha)$, we can combine them with Theorem~\ref{pt1} and Theoram~\ref{pt2} to get bounds on $M(\alpha)$.

For the family of all abelian groups, we have $M(\alpha) \le (27/26) \alpha^{3+c}$, where $c = 0.131 \dots$. Meanwhile, since $\alpha^4$ is a convex function of $\alpha$, we can deduce that for the family $G_n = \mathbb{F}_2^n$, $M(\alpha) \ge \alpha^4$. This last bound also holds for the family $G_n = \mathbb{F}_p^n$ for any prime $p$.

It would be natural to ask what is the correct exponent $\log(m(\alpha))/\log(\alpha))$. So far we only know that this exponent lies somewhere between $3.131$ and $4$. The author is continuing to study this question and hopes to report on further advances in due course.

\emph{Acknowledgements.} I am grateful to Ben Green for plenty of helpful comments and discussion.

\bibliographystyle{plain}
\bibliography{Bibliography}

\end{document}